\documentclass[a4paper,12pt]{article}
\usepackage{amsmath,amssymb,amsthm}

\theoremstyle{plain}
\newtheorem{thm}{Theorem}[section]
\newtheorem{lem}[thm]{Lemma}

\theoremstyle{definition}

\newtheorem{rem}[thm]{Remark}

\numberwithin{equation}{section}

\newcommand{\A}{\mathcal A} 

\newcommand{\Z}{\mathbb{Z}}

\newcommand{\bk}{\mathbf{k}}
	
\DeclareMathOperator{\dep}{dep}
\DeclareMathOperator{\Li}{Li}

\def\={\;=\;} \def\+{\,+\,} \def\m{\,-\,} \def\Eq{\;\,\equiv\;\,}

\begin{document}
\title{Analogues of the Aoki-Ohno and Le-Murakami relations for finite multiple zeta values}
\author{Masanobu Kaneko, Kojiro Oyama, and Shingo Saito}
\date{\today}
\maketitle 

\begin{abstract} We establish finite analogues of the identities known as the Aoki-Ohno relation and the Le-Murakami 
relation in the theory of multiple zeta values. We use an explicit form of a generating series 
given by Aoki and Ohno.
\end{abstract}

\section{Introduction and the statement}
For an index set of positive integers $\bk=(k_1,\ldots, k_r)$ with $k_1>1$, the multiple zeta value $\zeta(\bk)$ and the multiple zeta-star value $\zeta^\star(\bk)$
are defined respectively by the nested series
\begin{align*}
\zeta(\bk)&\=\sum_{m_1>\cdots>m_r>0}\frac{1}{m_1^{k_1}\cdots m_r^{k_r}} \\
\intertext{and}
\zeta^\star(\bk)&=\sum_{m_1\ge \cdots\ge m_r\ge1} \frac{1}{m_1^{k_1}\cdots m_r^{k_r}}.  
\end{align*}
We refer to the sum $k_1+\cdots +k_r$, the length $r$, and the number of components $k_i$ with $k_i>1$ as the 
weight, depth, and height of the index $\bk$ respectively.

For given $k$ and $s$, let $I_0(k,s)$ be the set of indices $\bk=(k_1,\ldots, k_r)$ with $k_1>1$ of weight $k$ and height $s$. 
We naturally have $k\ge 2s$ and $s\ge1$; otherwise $I_0(k,s)$ is empty.

Aoki and Ohno proved in \cite{AO} the identity
\begin{equation}\label{AOorig} \sum_{\bk\in I_0(k,s)} \zeta^\star(\bk)\= 2\binom{k-1}{2s-1}(1-2^{1-k})\zeta(k). \end{equation}
On the other hand, for $\zeta(\bk)$, the following identity is known as the Le-Murakami relation (\cite{LM}): For even $k$, it holds
\[ \sum_{\bk\in I_0(k,s)} (-1)^{\dep(\bk)} \zeta(\bk)\=\frac{(-1)^{k/2}}{(k+1)!}\sum_{r=0}^{k/2-s}\binom{k+1}{2r}(2-2^{2r})B_{2r} \pi^{k},\] 
where $B_n$ denotes the Bernoulli number. By Euler, the right-hand side is a rational multiple of the Riemann zeta value $\zeta(k)$.  \\

In this short article, we establish their analogous identities for {\em finite multiple zeta values}.   

For an index set of positive integers $\bk=(k_1,\ldots, k_r)$, the finite multiple zeta value $\zeta_\A(\bk)$ and the finite
multiple zeta-star value $\zeta_\A^\star(\bk)$ are elements in the quotient ring $\A:=\left(\prod_p \Z/p\Z\right) / \left(\bigoplus_p \Z/p\Z\right)$ 
($p$ runs over all primes) represented respectively by
\[ \left(\sum_{p>m_1>\cdots>m_r>0}\frac{1}{m_1^{k_1}\cdots m_r^{k_r}}\bmod p\right)_p 
\quad\text{and}\quad  \left(\sum_{p>m_1\ge\cdots\ge m_r>0}\frac{1}{m_1^{k_1}\cdots m_r^{k_r}}\bmod p\right)_p. \]
Studies of finite multiple zeta(-star) values go back at least to Hoffman \cite{H} (its preprint was available around 2004) and
Zhao \cite{Zh}. But it was rather recent that Zagier proposed (in 2012 to the first-named author) considering them in the (characteristic 0) 
ring $\A$ (\cite{KZ}, see \cite{K, K2}). 
In $\A$, the naive analogue $\zeta_{\A}(k)$ of the Riemann zeta value $\zeta(k)$ is zero because $\sum_{n=1}^{p-1} 1/n^k$ is congruent to 0
modulo $p$ for all sufficiently large primes $p$.  However, the ``true" analogue of $\zeta(k)$ in $\A$ is considered to be
\[ Z(k):=\left(\frac{B_{p-k}}{k}\right)_p. \]
We note that this value is zero when $k$ is even because the odd-indexed Bernoulli numbers are 0 except $B_1$. 
It is still an open problem whether $Z(k)\ne 0$ for any odd $k\ge3$. 

Now we state our main theorem, where the role of $Z(k)$ as a finite analogue of $\zeta(k)$ is evident.

\begin{thm}  The following identities hold in $\A$: 
\begin{eqnarray}
\sum_{\bk\in I_0(k,s)} \zeta_\A^\star(\bk)&\=&2\binom{k-1}{2s-1}(1-2^{1-k})Z(k),\label{AO}\\
\sum_{\bk\in I_0(k,s)} (-1)^{\dep(\bk)}\zeta_\A(\bk)&\=&2\binom{k-1}{2s-1}(1\m2^{1-k})Z(k)\label{LM}.
\end{eqnarray}
\end{thm}

We should note that the right-hand sides are exactly the same.
In the next section, we give a proof of the theorem.

\section{Proof}

Let $\Li_{\bk}^\star(t)$ be the ``non-strict" version of the multiple-polylogarithm:

\[ \Li_{\bk}^\star(t)\=\sum_{m_1\ge\cdots\ge m_r\ge1}\frac{t^{m_1}}{m_1^{k_1}\cdots{m_r}^{k_r}}.\]
Aoki and Ohno \cite{AO} computed the generating function
\[ \Phi_0\;:=\;\sum_{k,s\ge1}\left(\sum_{\bk\in I_0(k,s)}\Li_{\bk}^\star(t)\right)\,x^{k-2s}z^{2s-2},\]
and, in view of $\Li_{\bk}^\star(1)=\zeta^\star(\bk)$ (if $k_1>1$), evaluated it at $t=1$ to obtain the identity \eqref{AOorig}.  
For our purpose, the function $\Li_{\bk}^\star(t)$ is useful because the truncated sum
\[\sum_{p>m_1\ge\cdots\ge m_r\ge1}\frac{1}{m_1^{k_1}\cdots m_r^{k_r}} \]
used to define $\zeta_\A^\star(\bk)$ is the sum of the coefficients of $t^i$ in $\Li_{\bk}^\star(t)$ 
for $i=1,\ldots,p-1$.  In \S3 of \cite{AO}, they showed 
\[ \Phi_0=\sum_{n=1}^\infty a_n t^n, \]
where
\[ a_n\=\sum_{l=1}^n\left(\frac{A_{n,l}(z)}{x+z-l}+\frac{A_{n,l}(-z)}{x-z-l}\right)\]
and
\[ A_{n,l}(z)\=(-1)^l\binom{n-1}{l-1}\frac{(z-l+1)\cdots (z-1)z(z+1)\cdots(z+n-l-1)}
{(2z-l+1)\cdots (2z-1)2z(2z+1)\cdots(2z+n-l)}.\]
The problem is then to compute the coefficient of $x^{k-2s}z^{2s-2}$ in $\sum_{n=1}^{p-1} a_n$ modulo $p$.

We proceed as follows:
\begin{align*}
\sum_{n=1}^{p-1} a_n &= \sum_{n=1}^{p-1}\sum_{l=1}^{n}\left(\frac{A_{n,l}(z)}{x+z-l}+\frac{A_{n,l}(-z)}{x-z-l}\right)\\
&= \sum_{l=1}^{p-1}\sum_{n=l}^{p-1}\left(\frac{A_{n,l}(z)}{x+z-l}+\frac{A_{n,l}(-z)}{x-z-l}\right)\\
&= \sum_{l=1}^{p-1}\sum_{n=0}^{p-l-1}\left(\frac{A_{n+l,l}(z)}{x+z-l}+\frac{A_{n+l,l}(-z)}{x-z-l}\right).
\end{align*}
Writing $A_{n+l,l}(z)$ as 
\[ A_{n+l,l}(z)\=\frac{(-1)^l}{2z}\frac{(z-l+1)_{l-1}}{(2z-l+1)_{l-1}}\frac{(l)_n(z)_n}{(2z+1)_nn!}, \]
where $(a)_n=a(a+1)\cdots(a+n-1)$, we have
\[ \sum_{n=0}^{p-l-1} A_{n+l,l}(z)\=\frac{(-1)^l}{2z}\frac{(z-l+1)_{l-1}}{(2z-l+1)_{l-1}}\sum_{n=0}^{p-l-1}\frac{(l)_n(z)_n}{(2z+1)_nn!}. \]
We view the sum on the right as 
\[ \sum_{n=0}^{p-l-1}\frac{(l)_n(z)_n}{(2z+1)_nn!}\Eq F(-p+l,z;2z+1;1)-\frac{(l)_{p-l}(z)_{p-l}}{(2z+1)_{p-l}(p-l)!}\ \bmod p. \]
Here, $F(a,b;c;z)$ is the Gauss hypergeometric series
\[ F(a,b;c;z)\=\sum_{n=0}^\infty \frac{(a)_n(b)_n}{(c)_nn!} z^n, \]
where $(a)_n$ for $n\ge1$ is as before and $(a)_0=1$.
Note that if $a$ (or $b$) is a non-positive integer $-m$, then $F(a,b;c;z)$ is a polynomial in $z$ of degree at most $m$, and 
the renowned formula of Gauss 
\[ F(a,b;c;1)\=\frac{\Gamma(c)\Gamma(c-a-b)}{\Gamma(c-a)\Gamma(c-b)} \]
becomes 
\[ F(-m,b;c;1)\=\frac{(c-b)_m}{(c)_m}. \]
Hence 
\[ F(-p+l,z;2z+1;1)\=\frac{(z+1)_{p-l}}{(2z+1)_{p-l}}\equiv\frac{z^{p-1}-1}{(2z)^{p-1}-1}\frac{(2z-l+1)_{l-1}}{(z-l+1)_{l-1}}\ \bmod p. \]
We also compute 
\[ \frac{(l)_{p-l}(z)_{p-l}}{(2z+1)_{p-l}(p-l)!}\equiv (-1)^{l-1}\frac{z(z^{p-1}-1)}{(2z)^{p-1}-1}\frac{(2z-l+1)_{l-1}}{(z-l)_l}\  \bmod p. \]
Since we only need the coefficient of $z^{2s-2}$, we may work modulo higher power of $z$ and in particular 
we may replace $(z^{p-1}-1)/((2z)^{p-1}-1)$ by 1, assuming $p$ is large enough.
Hence we get
\begin{align*}  \sum_{n=1}^{p-1} a_n\equiv &\sum_{l=1}^{p-1}\left\{\frac{(-1)^l}{2z}\left(\frac1{x+z-l}-\frac1{x-z-l}\right) \right.\\
&\qquad \qquad+\left. \frac12\left(\frac1{(x+z-l)(z-l)}
-\frac1{(x-z-l)(z+l)}\right)\right\}\ \bmod p. 
\end{align*}
By the binomial expansion, we have
\begin{align*} 
\sum_{l=1}^{p-1} \frac{(-1)^l}{x+z-l}&\= \sum_{l=1}^{p-1}\frac{(-1)^{l-1}}{l}\sum_{m=0}^\infty \left(\frac{x+z}{l}\right)^m\\
&\=\sum_{l=1}^{p-1}\frac{(-1)^{l-1}}{l}\sum_{m=0}^\infty \frac1{l^m}\sum_{i=0}^m\binom{m}{i}x^{m-i}z^i\\
&\=\sum_{m\ge i\ge0}\binom{m}{i}\left(\sum_{l=1}^{p-1}\frac{(-1)^{l-1}}{l^{m+1}}\right) x^{m-i}z^i.
\end{align*}
From this we obtain
\[ \sum_{l=1}^{p-1}\frac{(-1)^l}{2z}\left(\frac1{x+z-l}-\frac1{x-z-l}\right) =
\sum_{m\ge 2i+1\ge0}\binom{m}{2i+1}\left(\sum_{l=1}^{p-1}\frac{(-1)^{l-1}}{l^{m+1}}\right) x^{m-2i-1}z^{2i}\]
and the coefficient of $x^{k-2s}z^{2s-2}$ in this is (by letting $i\to s-1$ and $m\to k-1$)
\[ \binom{k-1}{2s-1}\sum_{l=1}^{p-1}\frac{(-1)^{l-1}}{l^k}. \]
This is known to be congruent modulo $p$ to
\[ 2\binom{k-1}{2s-1}(1-2^{1-k})\frac{B_{p-k}}{k} \]
(see {\it e.g.} \cite[Theorem~8.2.7]{Zh2}). Concerning the other term
\begin{align*}
 &\sum_{l=1}^{p-1} \frac12\left(\frac1{(x+z-l)(z-l)}-\frac1{(x-z-l)(z+l)}\right) \\
 \=&\frac12\sum_{l=1}^{p-1} \left\{ \frac1x\left(\frac1{z-l}-\frac1{x+z-l}\right)- \frac1x\left(\frac1{z+l}+\frac1{x-z-l}\right)\right\},
 \end{align*}
every quantity that appears as coefficients in the expansion into power series in $x$ and $z$ is a multiple of the sum of the form
$\sum_{l=1}^{p-1} 1/l^m$, and are all congruent to 0 modulo $p$.  This concludes the proof of \eqref{AO}. 

We may prove \eqref{LM} in a similar manner by using the generating series of Ohno-Zagier \cite{OZ},
but we reduce \eqref{LM} to \eqref{AO} by showing that the left-hand sides of both formulas are equal up to sign.

Set $S_{k,s}:=\sum_{\bk\in I_0(k,s)} (-1)^{\dep(\bk)}\zeta_\A(\bk)$ and $S_{k,s}^\star := \sum_{\bk\in I_0(k,s)} \zeta_\A^\star(\bk)$.
\begin{lem}\label{zeta-zetastar}  $S_{k,s}^\star=(-1)^{k-1} S_{k,s}$. 
\end{lem}
\begin{proof}
We use the well-known identity (see for instance \cite[Corollary~3.16]{SS}) 
\begin{equation} \label{apid} \sum_{i=0}^r (-1)^i \zeta_\A(k_i,\ldots,k_1)\zeta_\A^\star(k_{i+1},\ldots,k_r)\=0. 
\end{equation}
Taking the sum of this over all $\bk\in I_0(k,s)$ and separating the terms corresponding to $i=0$ and $i=r$, we
obtain 
\[ S_{k,s}^\star+\sum_{k'+k''=k\atop s'+s''=s} \left(\sum_{\bk'\in I_0(k',s')} (-1)^{\dep(\bk')}\zeta_\A(\overleftarrow{\bk'})\right)
\left(\sum_{\bk''\in I(k'',s'')} \zeta_\A^\star(\bk'')\right)+(-1)^k S_{k,s}\=0. \]
Here, $\overleftarrow{\bk'}$ denotes the reversal of $\bk'$, and the set $I(k'',s'')$ consists of all indices (no restriction on the
first component) of weight $k''$ and height $s''$. 
We have used $\zeta_\A(\overleftarrow{\bk})=(-1)^k\zeta_\A(\bk)$ in computing the last term ($i=r$). Since the second
sum in the middle is symmetric and hence 0 (by using Hoffman \cite[Theorem~4.4]{H} and $\zeta_\A(k)=0$ for all $k\ge1$), we prove the lemma. 
\end{proof}
Since $Z(k)=0$ if $k$ is even, we see from Lemma \ref{zeta-zetastar} that the formula for $S_{k,s}$ is the same as that for $S_{k,s}^\star$.  
This concludes the proof of our theorem.

\begin{rem}  K.~Yaeo \cite{Y} proved the lemma in the case $s=1$ and T.~Murakami (unpublished) in general for all odd $k$. 
\end{rem}

\section{Acknowledgements}
The authors would like to thank Shin-ichiro Seki for his valuable comments on an earlier version of the paper.
This work was supported by JSPS KAKENHI Grant Numbers JP16H06336  and JP18K18712.

\end{document}